\documentclass{amsart}
\usepackage[a4paper, top=2cm, bottom=3cm, left=2.5cm, right=2.5cm]{geometry}
\usepackage[utf8]{inputenc}
\usepackage[english]{babel}
\usepackage{mathrsfs}
\usepackage{amsmath}
\usepackage{amssymb}
\usepackage{amsthm}
\usepackage{graphicx}
\usepackage{enumitem}
\usepackage{mathtools}
\usepackage{xcolor}
\usepackage{float}
\usepackage[colorlinks = true, urlcolor = blue, citecolor = green!30!gray, linkcolor = blue]{hyperref}
\usepackage{tikz-cd}
\usepackage{pgfplots}
\usepackage{tikz,tikz-3dplot}
\usepackage{thmtools}
\usepackage{csquotes}
\usepackage[toc,page]{appendix}
\usepackage[style=alphabetic, backend=biber, doi=false, url=false, giveninits=true,   maxnames=5, maxalphanames=3]{biblatex}
\DeclareNolabel{
  \nolabel{\regexp{[\p{P}\p{S}\p{C}\p{Z}]+}}
}
\DeclareFieldFormat[thesis]{title}{\emph{#1}}

\usepackage{longtable}
\usetikzlibrary{arrows,decorations.markings,intersections,patterns,shadings,calc,shapes,positioning }
\usetikzlibrary{shapes.geometric}
\usetikzlibrary{angles, quotes}
\usetikzlibrary{fadings}
\pgfplotsset{compat=newest}
\usepgfplotslibrary{fillbetween}
\numberwithin{equation}{section}
\mathtoolsset{showonlyrefs}
\setlist[description]{leftmargin=4mm}
\setlist[itemize]{leftmargin=5mm}
\setlist[enumerate]{leftmargin=6mm}
\theoremstyle{plain}
\newtheorem{teorema}{Theorem}[section]
\newtheorem{corollario}[teorema]{Corollary}

\newtheorem{lemma}[teorema]{Lemma}

\newtheorem*{teorema*}{Theorem}
\newtheorem*{corollario*}{Corollary}
\newtheorem*{proposizione*}{Proposition}
\newtheorem*{lemma*}{Lemma}
\newtheorem*{esercizio*}{Exercise}
\theoremstyle{definition}

\newtheorem*{definizione*}{Definition}
\newtheorem*{esempio*}{Example}
\newtheorem*{domanda*}{Question}
\theoremstyle{remark}
\newtheorem{osservazione}[teorema]{Remark}
\newtheorem*{osservazione*}{Remark}
\makeatletter
\renewenvironment{proof}[1][\proofname]{\par
  \pushQED{\qed}%
  \normalfont \topsep6\p@\@plus6\p@\relax
  \trivlist
  \item[\hskip\labelsep
        \bfseries
    #1\@addpunct{.}]\ignorespaces
}{%
  \popQED\endtrivlist\@endpefalse
}
\makeatother
\title{Stratification of the single blow-up set for Radon measures}

\author[L.~De Masi]{Luigi De Masi}
\address{\textit{L.~De Masi:} Dipartimento di Matematica, Università di Trento, Via Sommarive 14, 38123 Povo, Trento, Italy}
\email{luigi.demasi@unitn.it}

\begin{document}
	\begin{abstract}
		We show that the set of points where the blow-up, in the sense of Preiss, of a signed Radon measure on $\mathbb{R}^n$ is unique and its invariant subspace has dimension $k$ is $k$-rectifiable.
As applications, we obtain simple proofs of a rectifiability criterion for Radon measures and of a theorem, due to Mattila, on measures having unique blow-up almost everywhere.
	\end{abstract}
	\maketitle
%
%
%
%
%
%
%
%
%
\section{Introduction}
The notion of blow-up of a geometric or analytic object, i.e.\ the limit of its enlargements around a point as the zooming factor goes to infinity, is of fundamental importance in many areas of Mathematics. The blow-up of a function or a manifold provides a simplified, but somewhat essential, description of the local behavior of the object ``at infinitesimal scale".
For this and other reasons, the classification of the blow-ups of various mathematical items has become a central theme in geometric and functional analysis.

For a Radon measure $\mu$ in $\mathbb{R}^n$, a natural definition of blow-up at a point $x \in \mathbb{R}^n$ is the one in the sense of Preiss, \cite{preiss87}: the limit in the weak*-topology of Radon measures (in duality with $C_c(\mathbb{R}^n)$) of the rescalings
\begin{equation}\label{eq:rescalings}
c_j (\tau_{x,r_j})_\# \mu,
\end{equation}
for some sequences $c_j>0$, $r_j \to 0$, where $\tau_{x,r}(y)\coloneqq \frac{y-x}{r}$ and ${}_\#$ is the push-forward of Radon measures.

We say that $\mu$ has unique blow-up at $x \in \operatorname{supp} \mu$ if there exists a Radon measure $\mu_x$ with either $|\mu_x|(B_1(0)) = 1$ or $\mu_x = 0$ and such that, for every sequence $r_j \to 0^+$, there are a subsequence, not relabeled, and a constant $c\neq 0$ for which
\begin{equation}\label{eq:def-unique_blowup}
\mu_{x,r_j} \coloneqq
\frac{1}{|\mu|(B_{r_j}(x))} (\tau_{x,r_j})_\# \mu
\overset{\ast}{\rightharpoonup} c \mu_x,
\end{equation}
where $|\mu|$ is the total variation measure of $\mu$ and $\overset{\ast}{\rightharpoonup}$ denotes the weak*-convergence of Radon measures.
We will denote by $\mathcal{S}_\mu$ the set of points where $\mu$ has unique blow-up.

Mattila studied positive measures having a unique blow-up at almost every point: a way to rephrase his main result \cite[Theorem 3.2]{mattila_unique05} is that,
for $\mu$-a.e.\ $x \in \mathcal{S}_\mu$, the blow-up of $\mu$ at $x$ is a multiple of the $k$-dimensional Hausdorff measure $\mathcal{H}^k$ on a $k$-plane.
Although this describes the blow-ups of $\mu$ at $\mu$-a.e.\ point in $\mathcal{S}_\mu$, it can be interesting to study the behavior of $\mu$ in the remaining part of $\mathcal{S}_\mu$.

This problem has already received some attention: in \cite{delnin21} it is proved that, for a given $u \in L^1_{\text{loc}}(\mathbb{R}^n)$, the set $\Sigma_u$ of points $x$ where $\frac{1}{r^n}u \left (\frac{\cdot-x}{r} \right )$ converges, in the strong $L^1(B_1(0))$-topology, to a non-constant function can be covered by a countable union of $(n-1)$-dimensional Lipschitz graphs. 
\begin{itemize}

\item The results in \cite{delnin21} rely on the fact that the rescalings of $u$ converge strongly in $L^1(B_1(0))$, which is more restrictive than the convergence in the weak*-topology of Radon measures. Therefore, for $\mu=u \mathcal{L}^n$ where $\mathcal{L}^n$ is the Lebesgue measure, $\Sigma_u$ is in general strictly contained in the subset of $\mathcal{S}_\mu$ where the unique blow-up of $\mu$ is not a constant function.
Is the latter set rectifiable as well? What can be said if $\mu$ is a general Radon measure?

\item It is easily seen (Lemma \ref{lmm:0-hom}) that the unique blow-up $\mu_x$ of $\mu$ at a point in $x \in \mathcal{S}_\mu$ has an \emph{invariant linear subspace} $D_x$ of $\mathbb{R}^n$, namely $(\tau_{y,1})_\# \mu_x = \mu_x$ for $y \in D_x$. Therefore
\begin{equation}\label{eq:def_s_mu_k}
\mathcal{S}_\mu = \bigcup_{k=1}^n \mathcal{S}_\mu^k,
\qquad\text{where} \qquad
\mathcal{S}_\mu^k \coloneqq \left \{x \in \mathcal{S}_\mu \colon \dim D_x=k \right \}.
\end{equation}
Do these $\mathcal{S}_\mu^k$ enjoy finer rectifiability properties?
\end{itemize}

The main result of this paper provides an answer to these questions.

\begin{teorema}\label{thm:main}
Let $\mu$ be a signed Radon measure on $\mathbb{R}^n$ and $k \in \{0,1,\dots,n\}$. Then $\mathcal{S}_\mu^k$ is contained in a countable union of $k$-dimensional Lipschitz graphs. Moreover, the approximate tangent plane to $\mathcal{S}_\mu^k$ at $x$ coincides with $D_x$ for $\mathcal{H}^{k}$-a.e.\ $x \in \mathcal{S}_\mu^k$.
\end{teorema}

Theorem \ref{thm:main} generalizes the results of \cite{delnin21}, where the strong $L^1$-convergence of the rescalings can be actually replaced with the (weaker) weak*-convergence of the associated Radon measures. Indeed, for $u \in L^1_{\mathrm{loc}}(\mathbb{R}^n)$ and $\mu=u \mathcal{L}^n$, the set $\Sigma_u$ defined above is included in $\bigcup_{k=0}^{n-1}\mathcal{S}_\mu^k$. Each $\mathcal{S}_\mu^k$ is covered by a countable union of $k$-dimensional (and thus $(n-1)$-dimensional) Lipschitz graphs, by Theorem \ref{thm:main}.

\medskip

The idea of the proof of Theorem \ref{thm:main} is that, for every $x \in \mathcal{S}_\mu^k$, there is a radius $r>0$ for which the points in $\mathcal{S}_\mu^k \cap B_r(x)$ with similar properties are contained in a cone $x+C$ around $x+D_x$. Otherwise, by contradiction, one can find a sequence of such points $y_j \notin x+C$ with $r_j=|x-y_j| \to 0$ and such that the measures $c_j \mu_{y_j}$ and $\mu_{y_j,r_j}$ are as close as we want. Now, on one hand $c_j\mu_{y_j}$ converges to a measure $\nu$ whose invariant subspace is close to $D_x$; on the other hand $\mu_{y_j,r_j}$, being translations of $\mu_{x,r_j}$ with respect to points outside $C$, converge to a translation of $\mu_x$ whose invariant subspace is not contained in $C$, contradicting the closeness of $c_j\mu_{y_j}$ and $\mu_{y_j,r_j}$.

\medskip

A first outcome of Theorem \ref{thm:main} is the proof a rectifiability criterion for signed Radon measures in terms of their blow-ups .

\begin{corollario}\label{cor:rectif_criterion}
Let $\mu$ be a signed Radon measure on $\mathbb{R}^n$ and $k \in \{0,1,\dots,n\}$. Let us assume that, for $|\mu|$-a.e.\ $x \in \mathbb{R}^n$, there exists a Radon measure $\nu_x$ whose invariant linear subspace $V_x$ is $k$-dimensional and with the property that, for every sequence $r_j \to 0^+$ there are a subsequence, not relabeled, and a constant $c\neq 0$ such that
\begin{equation}
\frac{1}{r_j^k}  (\tau_{x,r_j})_\# \mu
\overset{\ast}{\rightharpoonup} c \nu_x.
\end{equation}
Then $\mu$ is $k$-rectifiable: there exists a signed function $\theta \in L^1_{\mathrm{loc}}(\mathcal{S}_\mu^k,\mathcal{H}^k)$ such that \mbox{$\mu= \theta \mathcal{H}^k \llcorner \mathcal{S}_\mu^k$,} where $\mathcal{S}_\mu^k$ is defined in \eqref{eq:def_s_mu_k} and is $k$-rectifiable by Theorem \ref{thm:main}.
\end{corollario}

The above statement is a generalization of the well-known rectifiability criterion for positive Radon measures, see for instance \cite[Theorem 11.8]{simonlectures83}, where one usually sets $\nu_x=\theta_0 \mathcal{H}^k \llcorner V_x$ for some $k$-dimensional linear subspace $V_x$ and assumes the existence of the limit $r^{-k}(\tau_{x,r_j})_\# \mu \overset{\ast}{\rightharpoonup} \theta_0 \mathcal{H}^k \llcorner V_x$ as $r \to 0^+$.

Our proof is based on the fact that Theorem \ref{thm:main} and the hypotheses of Corollary \ref{cor:rectif_criterion} easily imply that $\mu$ is concentrated on the $k$-rectifiable set $\mathcal{S}_\mu^k$ and that the upper $k$-density of $|\mu|$ is finite almost everywhere, from which it is straighforward to obtain $\mu \ll \mathcal{H}^k \llcorner \mathcal{S}_\mu^k$.
We nevertheless observe that the above criterion can be actually proved using results which are already present in literature: the almost-everywhere characterization of unique blow-ups given by \cite[Theorem 3.2]{mattila_unique05}, the Marstrand-Mattila rectifiability criterion \cite[Theorem 5.1]{delellis_rectifiability} and the almost-everywhere locality of blow-ups given by \cite[Proposition 3.12]{delellis_rectifiability} (which works for signed functions as well).

\medskip

A further consequence of Theorem \ref{thm:main} is the following alternative proof of \cite[Theorem 3.2]{mattila_unique05} for signed Radon measures. It is based on the intuitive idea that, since each $\mathcal{S}_{\mu}^k$ is $k$-rectifiable, it is reasonable to expect that its subset of points where the blow-up of $\mu$ is more diffused than a $k$-dimensional measure is $|\mu|$-negligible. As above, the result can be also obtained by applying the (signed version of) locality of blow-ups given by \cite[Proposition 3.12]{delellis_rectifiability} to  \cite[Theorem 3.2]{mattila_unique05}.

\begin{corollario}\label{cor:almost_everywhere_rect}
Let $\mu$ be a signed Radon measure on $\mathbb{R}^n$ and $k \in \{0,1,\dots,n\}$. Then, for $|\mu|$-a.e.\ $x \in \mathcal{S}_\mu^k$, the unique blow-up of $\mu$ at $x$ is a multiple of $\mathcal{H}^k\llcorner D_x$.
\end{corollario}

We conclude this introduction with a comment on the definition of unique blow-up at a point.
As already said above, we say that $\mu_x$ is the unique blow-up of $\mu$ at $x \in \operatorname{supp}\mu$ if, for every sequence $r_j \to 0$, \emph{there exists} a subsequence for which \eqref{eq:def-unique_blowup} holds true. This differs from the definition used in \cite{mattila_unique05} for positive Radon measures, where one does not require the convergence of a subsequence for every choice of $\{r_j\}_j$: according to \cite{mattila_unique05}, one says that  $\mu$ has unique blow-upat $x \in \mathbb{R}$ if there exists $\nu$ such that any non-trivial weak*-limit (in case it exists) of $c_j (\tau_{x,r_j})_\# \mu$, where $c_j>0$ and $r_j \to 0^+$, is of the form $c\nu$ for some $c>0$. 

\cite[Lemma 2.5]{mattila_unique05} shows, provided $\mu$ is positive, that this condition actually implies $\mu_{x,r} \overset{\ast}{\rightharpoonup} c\mu_x$ as $r \to 0$ and, subsequently, our definition (they are actually equivalent).
However, this does not hold in the case of signed Radon measures: let us consider two sequences of positive real numbers $a_j, \rho_j$ such that
\begin{equation}
\lim_{j \to +\infty} \frac{1}{a_j} \sum_{i=j+1}^{+\infty} a_i = 0,
\qquad
\lim_{j \to +\infty} \frac{\rho_{j+1}}{\rho_j}=0.
\end{equation}
Then the Radon measure on $\mathbb{R}$ defined as
\begin{equation}
\mu \coloneqq \sum_{j \in \mathbb{N}} a_j\left ( \delta_{2\rho_j + \rho_j^2} - \delta_{2\rho_j - \rho_j^2} \right )
\end{equation}
has the null measure as unique blow-up at $x=0$ in the sense of \cite{mattila_unique05}, but $\mu_{0,\rho_j} =\frac{1}{|\mu|(B_{\rho_j}(0))} (\tau_{0,\rho_j})_\# \mu $ has no converging subsequences because $|\mu_{x,\rho_j}|\big(B_2(0)\big) \to +\infty$ as $j\to +\infty$.

\subsection*{Acknowledgments}
The author would like to thank Andrea Marchese and Gian Paolo Leonardi for inspiring conversations. He is moreover grateful to Paolo Bonicatto for the support, the careful reading of a preliminary version of the manuscript and for highlighting the fact that Corollaries \ref{cor:rectif_criterion} and \ref{cor:almost_everywhere_rect} can be also obtained by applying the locality of blow-ups to the existent literature.

The author has been supported by PRIN project 2022PJ9EFL ``Geometric Measure Theory: Structure of Singular Measures, Regularity Theory and Applications in the Calculus of Variations", CUP: E53D23005860006 and by INdAM - GNAMPA Project, CUP E53C25002010001. Part of the work has been carried out during the thematic program on ``Free boundary problems" at ESI (Wien), which the author wishes to thank.

\section{Notation}\label{sec:notations}

Throughout the paper, $\mu$ denotes a fixed signed Radon measure on $\mathbb{R}^n$.

\begin{longtable}{@{\extracolsep{\fill}}lp{0.79\textwidth}}

\multicolumn{2}{c}{\textbf{General notation}}\\[4pt]
$B_r(x)$           			 			&   Open ball of radius $r$ and center $x$ in $\mathbb{R}^n$;\\
$\omega_k$								&	Lebesgue measure of the unit ball in $\mathbb{R}^k$;\\
$C(z,V,\gamma)$								& 	Open cone with center $z$, axis the $k$-dimensional linear subspace $V$ and opening angle $\gamma \in \left (0,\frac{\pi}{2} \right ]$: $\{y \in \mathbb{R}^n \colon \operatorname{dist}(y,z+V) <  (\sin \gamma)|y-z|\}$;\\
$\tau_{x,r}$							&  Translation with base $x$ and scaling factor $\frac{1}{r}$: $\tau_{x,r}(y)\coloneqq \frac{y-x}{r}$ for every $y \in \mathbb{R}^n$;\\

\multicolumn{2}{c}{\textbf{}}\\[-1pt]
\multicolumn{2}{c}{\textbf{Measures}}\\[4pt]
$\mathcal{L}^n$							&	Lebesgue measure in $\mathbb{R}^n$;\\
$\mathcal{H}^s$							&	$s$-dimensional Hausdorff measure in $\mathbb{R}^n$;\\
$|\mu|$									&	Total variation measure of the signed Radon measure $\mu$;\\
$\#$									&	Push-forward of Radon measures;\\
$\overset{\ast}{\rightharpoonup}$		&	Weak*-convergence of Radon measures, in duality with $C_c(\mathbb{R}^n)$;\\
$d$										&	Distance compatible with the weak*-convergence of Radon measures: \mbox{$\nu_j \overset{\ast}{\rightharpoonup} \nu \Longleftrightarrow d(\nu_j,\nu)\to 0$} and $\sup_{j \in \mathbb{N}} |\nu_j|(C)<+\infty$ for any bounded set $C \subset \mathbb{R}^n$.\\
$\mu_{x,r}$								&	Normalized rescaling of $\mu$ at $x \in \operatorname{supp} \mu$ at scale $r$: $\mu_{x,r} \coloneqq \frac{1}{|\mu|(B_r(x))}(\tau_{x,r})_{\#} \mu$;\\
$\mu_x$									& 	Normalized unique blow-up of $\mu$ at $x \in \mathbb{R}^n$ as in \eqref{eq:def-unique_blowup};\\ 
$D_x$									&	Invariant subspace of $\mu_x$, given by Lemma \ref{lmm:0-hom};\\
$\mathcal{S}_\mu$						&	Set of points in $\mathbb{R}^n$ where the blow-up of $\mu$ is unique;\\
$\mathcal{S}^k_\mu$						&	Set of points in $\mathbb{R}^n$ where the invariant subspace of the unique blow-up of $\mu$ has dimension $k$;\\
$\operatorname{Hom}_k^\beta$							& 	Set of measures $\nu$
such that for some $\alpha\in [0,\beta]$ it holds  $(\tau_{0,\lambda})_\# \nu=\lambda^\alpha \nu$ for all $\lambda>0$, whose invariant subspace has dimension at least $k$ and $|\nu|\big(B_1(0)\big)\leq 1$;
\end{longtable}

\section{Proof of the results}

We start with a simple remark on the definition of unique blow-up at a point: for every point $x \in \mathcal{S}_\mu$ such that $\mu_x \neq 0$, the constant $c$ in \eqref{eq:def-unique_blowup} is bounded and away from $0$.

\begin{lemma}\label{lmm:lower_bound_mass_blow-up}
Let $\mu$ be a signed Radon measure on $\mathbb{R}^n$ and let $\mu_x \neq 0$ be the unique blow-up of $\mu$ at $x \in \mathcal{S}_\mu$. Then there exists $\eta=\eta(x)>0$ such that, for every sequence $\mu_{x,r_j}$ converging to $c \mu_x$ as $r_j \to 0^+$, it holds $\eta \leq |c|\leq 1$.
\end{lemma}

\begin{proof}
The upper bound on $|c|$ trivially follows from the lower semi-continuity of the total variation measure of open sets under weak*-convergence and the fact that, by definition, it holds $|\mu_{x,r}|\big(B_1(0)\big) \leq 1$ for every $r>0$, while $|\mu_x| \big(B_1(0)\big) =1$. It remains to show the lower bound.

By contradiction, let us assume that, for every $m \in \mathbb{N}$, there exists a sequence $r_j(m) \to 0$ such that $\mu_{x,r_j(m)} \overset{\ast}{\rightharpoonup} c_m \mu_x$ with $|c_m| \leq \frac{1}{m}$. Since the distance $d$ metrizes the weak-*convergence of Radon measures, there exists $M_m \in \mathbb{N}$ such that
\begin{equation}
d(\mu_{x,r_j(m)}, c_m \mu_x) < \frac{1}{m}
\qquad
\forall  j \geq M_m.
\end{equation}
We can thus define a sequence $\rho_m  \downarrow 0$ such that $\rho_m \coloneqq r_{j_m}(m)$ for some $j_m \geq M_m$. Since the sequence of measures $c_m \mu_x$ converges strongly to $0$, it holds
\begin{equation}\label{eq:distance_from_0}
\lim_{m \to +\infty} d(\mu_{x,\rho_m}, 0)= 0.
\end{equation}
On the other hand, by definition of unique blow-up at $x$, the sequence $\mu_{x,\rho_m}$ must have a weakly*-converging subsequence, whose limit is $0$ by \eqref{eq:distance_from_0}, hence contradicting the assumption $c \neq 0$ in \eqref{eq:def-unique_blowup}.
\end{proof}

We now show that the unique blow-up of $\mu$ at $x$ is homogeneous and that it has an invariant subspace.

\begin{lemma}\label{lmm:0-hom}
Let $\mu$ be a signed Radon measure on $\mathbb{R}^n$ and let $\mu_x \neq 0$ be the unique blow-up of $\mu$ at $x \in \mathcal{S}_\mu$. Then there exists $\alpha=\alpha(x)\geq 0$ such that $\mu_x$ and $|\mu_x|$ are $(\alpha-n)$-homogeneous,\footnote{According to this definition, the measure $f\mathcal{L}^n$, where $f$ is $p$-homogenous and locally integrable, is $p$-homogeneous.} namely
\begin{equation}\label{eq:hom_blow-up}
(\tau_{0,\lambda})_\# \mu_x=\lambda^\alpha \mu_x,
\quad
(\tau_{0,\lambda})_\# |\mu_x|=\lambda^\alpha |\mu_x|
\qquad
\forall \lambda>0.
\end{equation}
For $\mu_x$, and more generally for every $(\alpha-n)$-homogeneous Radon measure, the set
\begin{equation}
C_x\coloneqq \left \{ y \in \mathbb{R}^n \colon (\tau_{y,\lambda})_\# \mu_x=\lambda^\alpha (\tau_{y,1})_\# \mu_x \,\,\,\,\forall \lambda>0 \right \}
\end{equation}
coincides with the invariant subspace $D_x$ of $\mu_x$, that is
\begin{equation}
D_x \coloneqq \left \{ y \in \mathbb{R}^n \colon (\tau_{y,1})_\# \mu_x= \mu_x \right \}.
\end{equation}
Moreover $D_x$ is a linear subspace of $\mathbb{R}^n$ and, provided $\mu_x \neq 0$, it holds $\dim D_x \leq \alpha$.
\end{lemma}

\begin{proof}
\begin{description}[font=\normalfont\itshape\space]
\item[$\bullet$ $\mu_x$ and $|\mu_x|$ are $(\alpha-n)$-homogeneous]
We have
\begin{equation}
(\tau_{0,\lambda})_\# \mu_{x,r}
=
\frac{1}{|\mu|\big(B_r(x)\big)} (\tau_{x,\lambda r})_{\#}\mu 
=
\frac{|\mu|\big(B_{\lambda r}(x)\big)}{|\mu|\big(B_{r}(x)\big)}
\mu_{x,\lambda r}
\qquad
\forall r,\lambda>0.
\end{equation}
Let us choose a sequence $r_j \to 0$ such that $\mu_{x,r_j} \overset{\ast}{\rightharpoonup} c\mu_x \neq 0$ and a subsequence such that $\mu_{x,\lambda r_{j_i}}$ converge as well to $c' \mu_x \neq 0$. Then the above equation implies
\begin{equation}\label{eq:limit_ratio_measures_balls}
(\tau_{0,\lambda})_\# c\mu_{x}
=
(c'q) \mu_x,
\qquad
\text{where } \,\,
q \coloneqq \lim_{i \to +\infty} \frac{|\mu|\big(B_{\lambda r_{j_i}}(x)\big)}{|\mu|\big(B_{r_{j_i}}(x)\big)}.
\end{equation}
This in particular proves that the product $c'q$ does not depend on the subsequence $r_{j_i}$ and the ratio $\frac{c'q}{c}$ depends only on $\lambda$; thus we can define the map $\lambda \mapsto \psi(\lambda)= \frac{c'q}{c}$ such that 
\begin{equation}
(\tau_{0,\lambda})_\# \mu_{x}
=
\psi(\lambda) \mu_x
\qquad
\forall \lambda>0.
\end{equation}
The fact that $\psi(1)=1$, $\psi(\lambda) \neq 0$ for every $\lambda>0$ and the continuity of the map $\lambda \mapsto (\tau_{0,\lambda})_\# \mu_{x}$ with respect to the weak*-topology imply $\psi(\lambda)>0$ for every $\lambda>0$.
Since for each open set $U \subseteq \mathbb{R}^n$ it holds
\begin{equation}
\begin{aligned}
(\tau_{0,\lambda})_\#|\mu_{x}|(U)
&=
\sup \left \{ \int f(y) \,\mathrm{d}(\tau_{0,\lambda})_\#\mu_x(y) \colon f \in C_c(U), \|f\|_{C^0} \leq 1 \right \}
\\
&=
\psi(\lambda)\sup \left \{ \int f(y) \,\mathrm{d}\mu_x(y) \colon f \in C_c(U), \|f\|_{C^0} \leq 1 \right \}
\\
&=
\psi(\lambda) |\mu_x|(U),
\end{aligned}
\end{equation}
we infer $(\tau_{0,\lambda})_\#|\mu_{x}|= \psi(\lambda)|\mu_x|$ for all $\lambda>0$ as well. As in \cite[Lemma 2.5 (2)]{mattila_unique05}, it is easily seen that $\psi(\lambda \rho)=\psi(\lambda)\psi(\rho)$, thus there exists $\alpha\geq 0$ such that $\psi(\lambda)=\lambda^\alpha$ for every $\lambda>0$.

\item[$\bullet$ $C_x=D_x$] We consider an $(\alpha-n)$-homogeneous Radon measure $\nu$ which is fixed throughout the rest of the proof and we prove the statements for the sets $C,D$ corresponding to $\nu$.
By definition of push-forward and of $\tau_{y,\lambda}$, it follows
\begin{equation}
\begin{gathered}
C= \{y \in \mathbb{R}^n \colon \nu(y+\lambda A) = \lambda^\alpha \nu(y+A) \,\,\,\,\forall \lambda>0,\,\,\forall A \subset \mathbb{R}^n \text{ bounded and Borel} \},
\\[4pt]
D= \{y \in \mathbb{R}^n \colon \nu(y+ A) = \nu(A) \,\,\,\,\forall A \subset \mathbb{R}^n \text{ bounded and Borel} \},
\end{gathered}
\end{equation}
where $y+\lambda A \coloneq \{y+\lambda z \colon z \in A\}=\tau_{y,\lambda}^{-1}(A)$.
Clearly $0 \in C,D$ and let $A \subset \mathbb{R}^n$ be any bounded Borel set. If $y \in C$, then
\begin{equation}
\nu(y+A)
=
\nu \left (2 \left (y+ \frac{-y+A}{2} \right )\right )
\overset{(0 \in C)}{=} 2^\alpha \nu \left (y+ \frac{-y+A}{2}\right )
\overset{(y \in C)}{=}
\nu (y-y+A)=
\nu(A),
\end{equation}
thus we infer $y \in D$. On the other hand, if $y \in D$, we have
\begin{equation}
\nu(y+\lambda A) \overset{(y \in D)}{=} \nu(\lambda A)
\overset{(0 \in C)}{=} \lambda^\alpha \nu(A)
\overset{(y \in D)}{=} \lambda^\alpha\nu(y+A)
\qquad
\forall \lambda>0,
\end{equation}
concluding $y \in C$.

\item[$\bullet$ $D_x$ is a linear space] For every $y,z \in D$, we have
\begin{equation}
\nu(y-z+ A) \overset{(y\in D)}{=} \nu(-z+A) \overset{(z\in D)}{=} \nu(z-z+A)
=
\nu(A),
\end{equation}
thus $y-z \in D$. For every $y \in D$ and $\lambda>0$, it holds
\begin{equation}
\nu(\lambda y + A) = \nu \left ( \lambda \left (y+ \frac{A}{\lambda} \right ) \right )
\overset{(0 \in C)}{=}
\lambda^\alpha \nu \left ( y + \frac{A}{\lambda}\right )
\overset{(y \in D)}{=}
\lambda^\alpha \nu \left ( \frac{A}{\lambda} \right )
\overset{(0 \in C)}{=}
\nu(A),
\end{equation}
showing $\lambda y \in D$.

\item[$\bullet$ The dimension of $D_x$ is at most $\alpha$]
Let us assume $\nu \neq 0$ and $k \coloneqq \dim D$.
Up to multiplying by a constant, we can assume $|\nu| \big(B_1(0)\big)=1$, so that the homogeneity provides
\begin{equation}\label{eq:hom_meas_balls}
|\nu|\big(B_r(0)\big)= r^\alpha\qquad \forall r>0.
\end{equation}
Let $Q \coloneqq [-1,1]^n$ be the unit cube. By a dyadic decomposition, for every $j \in \mathbb{N}$ the cube $Q$ contains $2^{kj}$ disjoint open balls of radius $2^{-j}$ with centers on $D$. Thus \eqref{eq:hom_meas_balls} and the invariance of $|\nu|$ with respect translations on $D$ yield
\begin{equation}
|\nu|(Q) \geq 2^{j(k-\alpha)} \qquad \forall j \in \mathbb{N},
\end{equation}
which implies $k\leq \alpha$.\qedhere
\end{description}
\end{proof}

\begin{osservazione}\label{rem:upper_bound_limsup_ratio}
In the first step of the above proof, \eqref{eq:limit_ratio_measures_balls}, $|c|,|c'| \in [\eta(x),1]$ and $\psi(\lambda)=\frac{c'q}{c}=\lambda^\alpha$ imply
\begin{equation}\label{eq:estimate_ratio_balls}
\lambda^\alpha \eta(x)
\leq
\liminf_{r \to 0} \frac{|\mu|\big(B_{\lambda r}(x)\big)}{|\mu|\big(B_{r}(x)\big)}
\leq
\limsup_{r \to 0} \frac{|\mu|\big(B_{\lambda r}(x)\big)}{|\mu|\big(B_{r}(x)\big)}
\leq
\frac{\lambda^\alpha}{\eta(x)}
\qquad
\forall \lambda>0.
\end{equation}
\end{osservazione}

\bigskip

We now prove that, if a Radon measure $\sigma$ is homogeneous with respect to a point which is sufficiently distant from the invariant subspace of another homogeneous measure $\nu$, then the two measures cannot be too close. See Section \ref{sec:notations} for the relevant notations.

\begin{lemma}\label{lmm:too_close}
For every $\beta,M\geq 0$ and every $\eta,\gamma,\delta>0$, there exists $\varepsilon>0$ with the following property.
Let $\nu,\sigma$ be signed Radon measures in $\mathbb{R}^n$ and let $V$ be a $k$-dimensional linear subspace such that:
\begin{itemize}
\item $|\nu|\big(B_1(0)\big) \leq 1$ and $|\sigma|\big(B_1(0)\big) \leq M$;
\item  $\nu$ is $(\alpha-n)$-homogeneous for $\alpha\in [0,\beta]$ and $V$ is its invariant subspace given by Lemma \ref{lmm:0-hom};
\item $d(c\nu,\omega)\geq \delta$ for every $|c| \in [\eta,1]$ and every $\omega \in \operatorname{Hom}_{k+1}^\beta$, that is the set of measures $\omega$ which are $(\alpha-n)$-homogeneous for some $\alpha\in [0,\beta]$, whose invariant subspace has dimension at least $k+1$ and such that $|\omega|\big(B_1(0)\big)\leq 1$;
\item  $(\tau_{y,\lambda})_\# \sigma=\lambda^{\alpha'} (\tau_{y,1})_\# \sigma$ for every $\lambda>0$, where $y \in \big( \overline{B_{\frac{1}{2}}(0)} \setminus B_{\frac{1}{4}}(0) \big)\setminus C(0,V,\gamma)$ and $\alpha' \in [0,\beta]$.
\end{itemize}
Then $d(c\nu,\sigma)>  \varepsilon$ for every $c \in [\eta,1]$.
\end{lemma}

\begin{proof}
Let us assume, by contradiction, that there exist $M,\beta\geq 0$ and $\eta,\gamma,\delta>0$ such that, for every $j \in \mathbb{N}$, there are measures $\nu_j,\sigma_j$, $k$-dimensional subspaces $V_j$ and $y_j \in \big(\overline{B_{\frac{1}{2}}(0)} \setminus B_{\frac{1}{4}}(0) \big)\setminus C(0,V_j,\gamma)$ satisfying the hypotheses of the statement with $\alpha_j,\alpha_j' \leq \beta$ and $d(c_j\nu_j,\sigma_j)\leq \frac{1}{j}$ for $|c_j| \in [\eta,1]$. Up to selecting a subsequence and without loss of generality, we can assume $c_j>0$ for every $j \in \mathbb{N}$.

The assumptions on $\nu_j$ and $\sigma_j$ imply that they are uniformly bounded on each bounded subset of $\mathbb{R}^n$; thus, up to subsequences, we can assume
\begin{gather}
\nu_j \overset{\ast}{\rightharpoonup} \nu,
\qquad
\sigma_j \overset{\ast}{\rightharpoonup} \sigma,
\\
\alpha_j \to \alpha \leq \beta,
\qquad
\alpha_j' \to \alpha' \leq \beta,
\qquad
c_j \to \bar{c} \in [\eta,1],
\qquad
V_j \to V,
\\
y_j \to y \in \big(\overline{B_{\frac{1}{2}}(0)} \setminus B_{\frac{1}{4}}(0) \big)\setminus C(0,V,\gamma),
\end{gather}
where the convergence of $V_j$ to $V$ is intended locally in the Hausdorff distance.
Since $d$ is continuous with respect to weak*-convergence, it holds $d(\bar{c}\nu,\sigma)=0$, that is $\bar{c}\nu=\sigma$. Moreover:
\begin{itemize}
\item The lower semicontinuity of the total variation measure of open sets with respect to weak*-convergence provides $|\nu|\big(B_1(0)\big) \leq 1$ and $|\sigma|\big(B_1(0)\big) \leq M$.
\item Again the continuity of $d$ with respect to weak*-convergence yields
\begin{equation}
d(c\nu,\omega)\geq \delta
\qquad
\forall |c| \in [\eta,1],\,\,\,\forall \omega \in \operatorname{Hom}_{k+1}^\beta.
\end{equation}
In particular $\nu \neq 0$.
\item Since for each $\lambda>0$ the functions $\tau_{y,\lambda}$ converge uniformly to $\tau_{z,\lambda}$ as $y \to z$, from $V_j \to V$ and $\nu_j \overset{\ast}{\rightharpoonup} \nu$, we have that $\nu$ is $(\alpha-n)$-homogeneous and that its invariant subspace $D$ contains $V$, thus $\dim D \geq k$. On the other hand, the above point implies $\dim D \leq k$, therefore $D=V$.
\item Since for each $\lambda>0$ the functions $\tau_{y_j,\lambda}$ converge uniformly to $\tau_{y,\lambda}$, we have $(\tau_{y,\lambda})_\# \nu=\lambda^{\alpha'} (\tau_{y,1})_\# \nu$. In order to infer $y \in V$, it remains to show that $\alpha'=\alpha$. To this aim, let us first observe that $|\nu|$ is $(\alpha-n)$-homogeneous with respect to the origin, as in the proof of Lemma \ref{lmm:0-hom}, and $(\alpha'-n)$-homogeneous with respect to $y$. Thus $B_{\lambda-|y|}(0) \leq B_\lambda(y) \subset B_{\lambda + |y|}(0)$ yields
\begin{equation}
(\lambda- |y|)^\alpha |\nu|\big(B_1(0)\big) \leq \lambda^{\alpha'} |\nu|\big(B_1(y)\big) \leq ( \lambda + |y|)^\alpha |\nu|\big(B_1(0)\big)
\qquad
\forall  \lambda>0,
\end{equation}
which, for large $\lambda$ and taking into account $\nu \neq 0$, implies $\alpha=\alpha'$.
\end{itemize}
The last point yields $y \in D=V$, which however contradicts $y \notin C(0,V,\gamma)$.
\end{proof}

We now proceed with the proof of Theorem \ref{thm:main}.

\begin{proof}[Proof of Theorem \ref{thm:main}]
The statement is trivial for $k=n$, thus we assume $k \in \{0,1,\dots,n-1\}$ be fixed throughout the proof.
Let us moreover fix $\gamma \in \left (0,\frac{1}{16}\right )$. It is possible to choose a finite number $m$ of $k$-dimensional linear subspaces $\{V_\ell\}_{\ell=1}^m$ such that every $k$-dimensional linear subspace $V$ is contained in $\overline{C(0,V_\ell,\gamma)}$ for some $\ell \in \{1,\dots,m\}$. For every $\ell \in \{1,\dots,m\}$ and every $i \in \mathbb{N}$, we define the set
\begin{equation}
E_{\ell,i} \coloneqq \left \{ x \in \mathcal{S}^k \colon D_x \subset \overline{C(0,V_\ell,\gamma)},\,\,\, \alpha(x) \leq i,\,\,\, \eta(x)\geq \frac{1}{i},\,\,\, d(c\mu_{x},\omega)\geq \frac{1}{i} \,\,\,\forall |c| \in \left [\frac{1}{i},1 \right ],\, \forall\omega \in \operatorname{Hom}_{k+1}^i \right \},
\end{equation}
where $D_x$ and $\alpha(x)$ are respectively the invariant subspace and the homogeneity exponent of $\mu_x$ given by Lemma \ref{lmm:0-hom}, $\eta=\eta(x)$ is given by Lemma \ref{lmm:lower_bound_mass_blow-up}, while $d$ and $\operatorname{Hom}_{k+1}^i$ are respectively the distance and the set of $(\alpha-n)$-homogeneous measures described in Section \ref{sec:notations}.
Since $\mathcal{S}^k= \bigcup_{i \in \mathbb{N}}\bigcup_{\ell=1}^m E_{i,\ell}$, it is enough to prove that each $E_{i,\ell}$ can be covered by a countable union of Lipschitz graphs.

Let us fix $i \in \mathbb{N}$ and $\ell \in \{1,\dots,m\}$. 
Then fix the value $\varepsilon$ given by Lemma \ref{lmm:too_close} for $\beta=i$, $\delta=\frac{1}{i}$, $M=i \cdot 2^\alpha$ and, for every $\rho>0$, define the set
\begin{equation}
E_{i,\ell,\rho} \coloneqq
\left \{x \in E_{i,\ell} \colon \inf_{|c| \in [\eta,1]}d(\mu_{x,r},c\mu_x)< \varepsilon \,\,\,\forall r \in (0,\rho) \right \}.
\end{equation}
If $x \in E_{i,\ell}$, then it belongs to $E_{i,\ell,\rho}$ for some $\rho>0$, because the existence of a sequence $r_j \to 0$ such that $d(\mu_{x,r_j},c\mu_x) \geq \varepsilon$ for every $j \in \mathbb{N}$ and every $|c| \in [\eta,1]$ would contradict $x \in \mathcal{S}$. Therefore
\begin{equation}
E_{i,\ell} = \bigcup_{\rho \in \mathbb{Q} \cap (0,+\infty)} E_{i,\ell,\rho}
\end{equation}
and it suffice to prove that, for a fixed $\rho>0$, the set $E_{i,\ell,\rho}$ can be covered with a countable union of Lipschitz graphs.
In order to do so, we claim that, for every $x \in E_{i,\ell,\rho}$, there exists $r>0$ such that
\begin{equation}\label{eq:outside_cones}
E_{i,\ell,\rho} \cap B_{r}(x) \setminus C(x,D_x,3\gamma)= \{x\}.
\end{equation}
Let us fix $x\in E_{i,\ell,\rho}$, let $\mu_x$ be $(\alpha-n)$-homogeneous and let us assume, toward a contradiction, that for every $j \in \mathbb{N}$ there exists 
\begin{equation}\label{eq:y_j_final}
y_j \in E_{i,\ell,\rho} \cap  \overline{B_{\frac{r_j}{2}}(x)} \setminus\Big( B_{\frac{r_j}{4}}(x) \cup C(x,D_x,3\gamma) \Big),
\end{equation}
for some $r_j \leq \rho$ such that $r_j \downarrow 0$. This in particular implies the existence of constants $c_j$ with $|c_j| \in \left [\frac{1}{i},1 \right ]$ such that
\begin{equation}\label{eq:distance_muj}
d( c_j\mu_{y_j}, \mu_{y_j,r_j})< \varepsilon
\qquad
\forall j \in \mathbb{N}.
\end{equation}
Up to selecting a non relabeled subsequence and without loss of generality, we can assume $c_j>0$ and that $\mu_{x,r_j} \overset{\ast}{\rightharpoonup} c\mu_x$ for some $c \in \left [\frac{1}{i},1 \right ]$.
We now analyze the weak*-limit (up to a suitable subsequence) of the two sequences of measures in \eqref{eq:distance_muj}.

\begin{itemize}
\item 
The first component $c_j\mu_{y_j}$ of \eqref{eq:distance_muj} satisfies the homogeneity of Lemma \ref{lmm:0-hom} with $c_j \in \left [\frac{1}{i},1 \right ]$, $\alpha(y_j)\leq i$ and $|\mu_{y_j}|\big(B_1(0)\big)= 1$; therefore the sequence of measures $\mu_{y_j}$ is weakly relatively compact and converges, up to a subsequence, to a Radon measure $\nu$, which, by the same arguments in the proof of Lemma \ref{lmm:too_close}, is $(\alpha'-n)$-homogeneous for some $\alpha' \leq i$, its invariant subspace $V$ has dimension $k$ and
\begin{equation}\label{eq:nu_sat_hyp_lemma}
|\nu|\big(B_1(0)\big)\leq 1,
\qquad
d(t\nu, \omega)>\frac{1}{i} \quad \forall |t| \in \left [ \frac{1}{i},1 \right ],\,\,\,\forall \omega \in \operatorname{Hom}_{k+1}^i.
\end{equation}
Moreover
\begin{equation}
D_{\mu_{y_j}} \subset \overline{C(0,V_\ell,\gamma)} \subset \overline{C(0,D_x,2\gamma)} \quad \forall j \in \mathbb{N}
\implies
V \subset \overline{C(0,D_x,2\gamma)}.
\end{equation}

\item
The second component $\mu_{y_j,r_j}$ of \eqref{eq:distance_muj} satisfies
\begin{equation}\label{eq:rewriting_transl}
\mu_{y_j,r_j}
= 
\frac{|\mu| \big(B_{r_j}(x)\big)}{|\mu| \big(B_{r_j}(y_j)\big)}  \left (\tau_{\frac{y_j-x}{r_j},1} \right )_\# \mu_{x,r_j}.
%
%
%
%
\end{equation}
Since  $B_{\frac{r_j}{2}}(x) \subset B_{r_j}(y_j)$, we have
\begin{equation}
\frac{|\mu| \big(B_{r_j}(x)\big)}{|\mu| \big(B_{r_j}(y_j)\big)}
\leq
\frac{|\mu| \big(B_{r_j}(x)\big)}{|\mu| \big(B_{\frac{r_j}{2}}(x)\big)},
\end{equation}
where, by \eqref{eq:estimate_ratio_balls}, it holds
\begin{equation}
\limsup_{j \to +\infty}
\frac{|\mu| \big(B_{r_j}(x)\big)}{|\mu| \big(B_{\frac{r_j}{2}}(x)\big)} \leq i\cdot 2^\alpha.
\end{equation}
Taking into account also \eqref{eq:y_j_final}, up to a subsequence we thus have
\begin{equation}
\frac{|\mu| \big(B_{r_j}(x)\big)}{|\mu| \big(B_{r_j}(y_j)\big)}
\to q \leq i\cdot 2^\alpha,
\qquad
\frac{y_j-x}{r_j} \to y \in  \overline{B_{\frac{1}{2}}(0)} \setminus \big( B_{\frac{1}{4}}(0)\cup C(0,D_x,3\gamma) \big).
\end{equation}
Therefore, \eqref{eq:rewriting_transl} and the convergence $\mu_{x,r_j} \overset{\ast}{\rightharpoonup} c \mu_x$ yield
\begin{equation}
\mu_{y_j,r_j} 
\overset{\ast}{\rightharpoonup}
cq (\tau_{y,1})_{\#} \mu_x.
\end{equation}
\end{itemize} 
Passing to the limit in \eqref{eq:distance_muj}, we have
\begin{equation}\label{eq:nu_close_mux_trasl}
d\big(\nu, c q (\tau_{y,1})_{\#} \mu_x \big) \leq \varepsilon.
\end{equation}
We now observe that, by \eqref{eq:nu_sat_hyp_lemma}, $\nu$ satisfies the assumptions of Lemma \ref{lmm:too_close} with $\beta=i$ and $\delta=\frac{1}{i}$, whereas the measure $\sigma\coloneqq c q (\tau_{y,1})_{\#} \mu_x$ satisfies $|\sigma| \big(B_1(0)\big) \leq i\cdot 2^\alpha$ and it is $(\alpha-n)$-homogeneous with respect to $-y$, because
\begin{equation}
(\tau_{-y,\lambda})_\# \sigma = c(\tau_{0,\lambda})_\# \mu_x = c\lambda^\alpha \mu_x
=
\lambda^\alpha (\tau_{-y,1})_\# \sigma.
\end{equation}
We moreover point out that, since $y \notin C(0,D_x,3\gamma)$ while $V \subset \overline{C(0,D_x,2\gamma)}$, it holds $y \notin C(0,V,\gamma)$.
Hence $\sigma$ satisfies the hypothesis of Lemma \ref{lmm:too_close} as well; however its conclusion is in contradiction with \eqref{eq:nu_close_mux_trasl}.

This proves the existence of $r$ such that \eqref{eq:outside_cones} holds true and \cite[Lemma 15.13]{mattilageometry95} (see also the observation below its proof) implies that $E_{i,\ell,\rho}$ can be covered by a countable union of Lipschitz graphs.

From \eqref{eq:outside_cones} we also infer that, for $\mathcal{H}^k$-a.e.\ $x \in E_{i,\ell,\rho}$, the approximate tangent plane $T_x E_{i,\ell,\rho}$ is contained in $\overline{C(0,D_x,3\gamma)}$. Since $\gamma$ can be chosen arbitrarily small, we have
\begin{equation}
T_x E_{i,\ell,\rho}=D_x
\qquad
\text{for $\mathcal{H}^k$-a.e.\ $x \in E_{i,\ell,\rho}$.}
\end{equation}
By standard arguments it holds $T_x E_{i,\ell,\rho}=T_x \mathcal{S}^k$ for $\mathcal{H}^k$-a.e.\ $x \in E_{i,\ell,\rho}$, which proves the last assertion of the statement.
\end{proof}

We now prove Corollary \ref{cor:rectif_criterion}.

\begin{proof}[Proof of Corollary \ref{cor:rectif_criterion}]
We first of all observe that, for every $x \in \mathbb{R}^n$ satisfying the hypotheses of the statement, it holds
\begin{equation}
\sup_{r \in (0,1]} \frac{|\mu|\big(B_r(x)\big)}{r^k} \eqqcolon s(x)
<
+\infty,
\end{equation}
otherwise there would be a sequence $r_j \to 0$ for which the above ratio is unbounded, contradicting the existence of a converging subsequence of $r_j^{-k} (\tau_{x,r_j})_\# \mu$. Since this happens $|\mu|$-a.e., we have
\begin{equation}
|\mu| = |\mu| \llcorner \mathcal{S}_\mu^k
\end{equation}
and $\mathcal{S}_\mu^k$ is $k$-rectifiable by Theorem \ref{thm:main}. We now show that $|\mu| \ll \mathcal{H}^k \llcorner E$. Indeed, let us consider $A \subset \mathcal{S}_\mu^k$ with $\mathcal{H}^k (A)= 0$. Then
\begin{equation}
A= \bigcup_{i \in \mathbb{N}} A_i,
\qquad
A_i \coloneqq \{ x \in A \colon s(x) \leq i\}
\end{equation}
and let us fix $i \in \mathbb{N}$. Since $\mathcal{H}^k(A_i)=0$, for every $\varepsilon>0$ there exists a countable covering $\{B_{r_\ell}(x_\ell)\}_{\ell \in \mathbb{N}}$ of $A_i$ such that\footnote{According to the definition of $\mathcal{H}^k$, the balls which cover $A$ are not necessarily centered on $A$, but $x \in A_i \cap B_r(y)$ implies $B_r(y) \subset B_{2r}(x)$.}
\begin{equation}
\sum_{\ell \in \mathbb{N}} r_\ell^k < \varepsilon
\qquad
\text{and}
\qquad
x_\ell \in A_i
\quad
\forall \ell \in \mathbb{N}.
\end{equation}
Thus we have
\begin{equation}
|\mu| (A_i)
\leq
\sum_{\ell \in \mathbb{N}} |\mu| \big(B_{r_\ell}(x_\ell) \big)
\leq
i \sum_{\ell \in \mathbb{N}} r_\ell^k
<
i \cdot \varepsilon.
\end{equation}
The arbitrariness of $\varepsilon$ implies $|\mu|(A_i)=0$. Since this is true for every $i \in \mathbb{N}$, we infer $|\mu|(A)=0$, hence $|\mu| \ll \mathcal{H}^k \llcorner \mathcal{S}_\mu^k$. Recalling that $|\mu|$ is a Radon measure and $\mathcal{H}^k \llcorner \mathcal{S}_\mu^k$ is $\sigma$-finite, the Radon-Nikodym Theorem implies the existence of a function $\theta \in L^1_{\mathrm{loc}}(\mathcal{S}_\mu^k,\mathcal{H}^k)$ such that $\mu= \theta \mathcal{H}^k \llcorner \mathcal{S}_\mu^k$.
\end{proof}

We conclude with the proof of Corollary \ref{cor:almost_everywhere_rect}.

\begin{proof}[Proof of Corollary \ref{cor:almost_everywhere_rect}]
From Theorem \ref{thm:main} we know that $\mathcal{S}_\mu^k$ is contained in the union of a family $\{M_i\}_{i \in \mathbb{N}}$ of $k$-dimensional Lipschitz graphs. Since $|\mu|$-a.e.\ $x \in M_i$ is a point of density $1$ for $M_i$ with respect to $|\mu|$, the blow-ups of $\mu$ and $\mu \llcorner M_i$ at $x$ are the same, therefore
\begin{equation}
\mu_x= (\mu \llcorner M_i)_x
\qquad
\text{for $|\mu|$-a.e.\ $x \in M_i$.}
\end{equation}
For such points, this implies that $\operatorname{supp}\mu_x$ is contained in a closed cone $C$ around a $k$-dimensional linear subspace $V$; we observe that $C$ does not contain any half-space of dimension larger than $k$. On the other hand, from Lemma \ref{lmm:0-hom} we infer the implication
\begin{equation}
y \in \operatorname{supp} \mu_x \implies \{\lambda y + z \colon \lambda>0, z \in D_x\} \subseteq  \operatorname{supp} \mu_x.
\end{equation}
Then it must hold either $\mu_x=0$ or $\operatorname{supp} \mu_x = D_x$. The invariance of $\mu_x$ on $D_x$ implies the conclusion.
\end{proof}

\section*{Conflicts of interest and data}
The author declares that he has no competing interests and that the manuscript does not have any associated data.

%
%
%
%
%
%
%
\printbibliography
%
%
%
%
%
%
%
%
%
%
%
%
%
%
\end{document}